\documentclass{article}
\usepackage{amssymb,amsmath,amsthm, amscd, verbatim,mathrsfs}
\usepackage{txfonts} % \Perp
\usepackage{color,hyperref}

\newtheorem{thm}{Theorem}%[section]

\newtheorem{lem}[thm]{Lemma}
\theoremstyle{definition}
\newtheorem{df}[thm]{Definition}
\theoremstyle{remark}
\newtheorem{exa}[thm]{Example}
\newtheorem{rem}[thm]{Remark}
\newtheorem{que}[thm]{Question}

\newcommand{\Iff}{\Leftrightarrow}\renewcommand{\iff}{\leftrightarrow}
\newcommand{\Implies}{\Rightarrow}

\newcommand{\mc}{\mathcal}

\newcommand{\restrict}{\upharpoonright}
\DeclareMathOperator{\Auto}{Aut}
\newcommand{\Aut}{\Auto(\mathcal D)}

\newcommand{\abs}[1]{|#1|}

\begin{document}
	\title{A tractable case of the Turing automorphism problem: bi-uniformly $E_0$-invariant Cantor homeomorphisms}
	\author{Bj\o rn Kjos-Hanssen\thanks{
		This work was partially supported by grants from the Simons Foundation (\#315188 and \#704836 to Bj\o rn Kjos-Hanssen) and by the
		Institute for Mathematical Sciences, National University of Singapore.
	}}
	\maketitle

	\begin{abstract}
		A function $F:2^\omega\to 2^\omega$ is an $E_0$-isomorphism if for all $x,y\in 2^\omega$, we have $xE_0y\iff f(x)E_0 f(y)$, where $xE_0y\iff(\exists a)(\forall n\ge b) x(n)=y(n)$.
		If such witnesses $a$ for $xE_0 y$ and for $f(x)E_0 f(y)$ depend on each other but not on $x$, $y$, then $F$ is called bi-uniform.
		It is shown that a homeomorphism of Cantor space which is a bi-uniform $E_0$-isomorphism can induce only the trivial automorphism of the Turing degrees.
	\end{abstract}

	\begin{center}
		\emph{
			Dedicated to the celebration of the work of Theodore A. Slaman and W. Hugh Woodin
		}
	\end{center}

	\tableofcontents
	\section{Introduction}

		Let $\mathscr D_{\mathrm{T}}$ denote the set of Turing degrees and let $\le$ denote its ordering.
		This article gives a partial answer to the following famous question.
		\begin{que}\label{rigid}
			Does there exist a nontrivial automorphism of $\mathscr D_{\mathrm{T}}$?
		\end{que}
		\begin{df}
			A bijection $\pi:\mathscr D_{\mathrm{T}}\to\mathscr D_{\mathrm{T}}$ is an
			\emph{automorphism} of $\mathscr D_{\mathrm{T}}$ if
			for all $\mathbf x, \mathbf y\in\mathscr D_{\mathrm{T}}$,
			$\mathbf x\le\mathbf y$ iff $\pi(\mathbf x)\le\pi(\mathbf y)$.
			If moreover there exists an $\mathbf x$ with $\pi(\mathbf x)\ne\mathbf x$ then $\pi$ is \emph{nontrivial}.
		\end{df}
		Question~\ref{rigid} has a long history.
		Already in 1977, Jockusch and Solovay~\cite{MR0432434} showed that
		each jump-preserving automorphism of the Turing degrees is the identity above $\mathbf 0^{(4)}$.
		Nerode and Shore~\cite{Nerode.Shore:80} showed that
		each automorphism (not necessarily jump-preserving) is equal to the identity on some cone $\{\mathbf a:\mathbf a\ge\mathbf b\}$.
		Slaman and Woodin~\cite{MR2449478,SW}
		showed that each automorphism is equal to the identity on the cone above $\mathbf 0''$
		and that $\Aut$ is countable.

		There is an obstacle to reducing the base of the cone to $\mathbf 0'$ and ultimately $\mathbf 0$:
		Turing reducibility is $\Sigma^0_3$, but not $\Pi^0_2$ or $\Sigma^0_2$ in the sense of descriptive set theory.

		In the other direction, S.~Barry Cooper~\cite{Cooper} claimed to construct a nontrivial automorphism,
		induced by a discontinuous function on $\omega^\omega$, itself induced by a function on $\omega^{<\omega}$.
		That claim was not independently verified.
		In~\cite{MR3840751} we attacked the problem by ruling out a certain simple but natural possibility: automorphisms induced by permutations of finite objects.
		We showed that no permutation of $\omega$ represents a nontrivial automorphism of the Turing degrees.
		That proof was too complicated, in a way, and did not extend from $D_T$ down to $D_m$. Here we give a more direct proof using the shift map $n\mapsto n+1$.
		Our proof here will generalize to a certain class of homeomorphisms,
		distinct from the class of such homeomorphisms that the result in~\cite{MR3840751} generalizes to.

	\section{Excluding permutations by recursion}

	\begin{lem}\label{waffle}
		Suppose $\theta:\omega\to\omega$ is a bijection such that
		$\theta^{-1}\circ S\circ\theta$ is computable, where $S$ is the successor function given by $S(n)=n+1$.
		Then $\theta$ is computable.
	\end{lem}
	\begin{proof}
		Let $k(n)=(\theta^{-1}\circ S\circ\theta)(n)$.
		Then for any $m$, $k(\theta^{-1}(m))=\theta^{-1}(S(m))$ and so we compute $\theta^{-1}$ by recursion:
		\[
			\theta^{-1}(m+1) = k(\theta^{-1}(m)).\qedhere
		\]
	\end{proof}

	We write $\sigma\preceq\tau$ if $\sigma$ is a prefix of $\tau$.

	\begin{lem}\label{1111}
		Suppose $\sigma\in 2^{<\omega}$, $g:\omega\to\omega$, and $\Phi$ is a Turing functional, satisfying
		\begin{equation}\label{rochester}
			( \forall n)(\forall \tau\succeq\sigma)(\exists\rho\succeq\tau)(\Phi^\rho(n)\downarrow \text{ and } \rho\circ g(n) = \Phi^\rho(n)).
		\end{equation}
		For any $\rho\succeq\sigma$ and $n$, if $\Phi^\rho(n)\downarrow$, then $g(n)<|\rho|$.
	\end{lem}
	\begin{proof}
		If instead $g(n)\ge |\rho|$ then $\rho(g(n))$ is undefined. So let $\tau\succeq\rho$, $\tau(g(n))=1-\Phi^{\rho}(n)$.
		This $\tau$ violates (\ref{rochester}).
	\end{proof}
	\begin{lem}\label{rumble}
		If $g:\omega\to\omega$ is injective and $\Phi$ is a Turing functional such that
		\[
			\{B: B\circ g = \Phi^B\}
		\]
		is nonmeager, then $g$ is computable.
	\end{lem}
	\begin{proof}
		By assumption, it is not the case that
		\[
			(\forall \sigma)( \exists n)(\exists \tau\succeq\sigma)(\forall\rho\succeq\tau)(\Phi^\rho(n)\downarrow \to \rho\circ g(n) \ne \Phi^\rho(n)).
		\]
		So we have
		\[
			(\exists \sigma)( \forall n)(\forall \tau\succeq\sigma)(\exists\rho\succeq\tau)(\Phi^\rho(n)\downarrow \text{ and } \rho\circ g(n) = \Phi^\rho(n)).
		\]
		Pick such a $\sigma$: then $\Phi$ cannot make a mistake above $\sigma$, and we can always extend to get the right answer.

		As finite data we assume we know the values of $n$ and $g(n)$ for which $g(n)<|\sigma|$.

		We compute the value $g(n)$ as follows.

		Check the finite database of $\{(k,g(k)):g(k)<|\sigma|\}$, and output $g(n)$ if found. Otherwise we know $g(n)\ge|\sigma|$.

		By dovetailing computations, find a $\rho_0\succeq\sigma$ such that $\Phi^{\rho_0}(n)\downarrow$.
		By Lemma~\ref{1111} we have that
		$g(n)<|\rho_0|$.
		Thus, $g(n)\in I$ where $I$ is the closed interval $[|\sigma|,|\rho_0|-1]$.
		Let $a\in I, b\in I, a<b$.
		It suffices to show how to eliminate either $a$ or $b$ as a candidate for being equal to $g(n)$.

		Let $\tau\succ\sigma$ be such that $\tau(a)\ne\tau(b)$ and let $\rho\succeq\tau$ be such that $\Phi^\rho(n)\downarrow$.
		Then mark as eliminated whichever $c\in\{a,b\}$ makes $\rho(c)\ne \Phi^\rho(n)$.
		Thus we one-by-one eliminate all $a\in I$
		until only one candidate remains.
	\end{proof}

	\begin{df}
	We say that a permutation $\theta:\omega\to\omega$ \emph{induces} the automorphism $\pi$ of $\mc D_r$ if, letting $\Pi=\Pi_{\theta}$ be defined by
	$\Pi(A)=A\circ\theta$ for all $A$,
	we have $\pi([X]_r)=[\Pi(X)]_r$ for all $X$.
	\end{df}
	\begin{lem}\label{fromReferee}
		If $\theta$ is a permutation of $\omega$ and induces an automorphism $\pi$ of $\mc D_r$ then $\theta^{-1}$ induces $\pi^{-1}$.
	\end{lem}
	\begin{proof}
		We must show that $\pi^{-1}([X]_r)=[\Pi_{\theta^{-1}}(X)]_r$.
		We have $\Pi_{\theta}^{-1}=\Pi_{\theta^{-1}}$, since
		\[
			\Pi_{\theta^{-1}}(\Pi_{\theta}(X)) = (X\circ\theta)\circ\theta^{-1} = X= (X\circ\theta^{-1})\circ\theta = \Pi_{\theta}(\Pi_{\theta^{-1}}(X)).
		\]
		We have the well-definedness condition
		\(
			X\equiv_r Y \Implies \Pi(X)\equiv_r \Pi(Y).
		\)
		Moreover, since $\pi$ is injective, it must be that
		\(
			X\equiv_r Y \Iff \Pi(X)\equiv_r \Pi(Y).
		\)
		Since $\Pi_{\theta}^{-1}=\Pi_{\theta^{-1}}$, we also have
		\[
			\Pi^{-1}(X)\equiv_r \Pi^{-1}(Y) \Iff X\equiv_r Y\iff
			\Pi_{\theta^{-1}}(X)\equiv_r \Pi_{\theta^{-1}}(Y),
		\]
		so that $\theta^{-1}$ also induces the automorphism $\pi^{-1}$:
		\begin{eqnarray*}
			[\Pi^{-1}(Y)]_r = [X]_r&\Iff& \Pi^{-1}(Y)\equiv_r X\Iff Y\equiv_r \Pi(X)\\
			\Iff [Y]_r = [\Pi(X)]_r &\Iff& [Y]_r=\pi([X]_r)\Iff \pi^{-1}([Y]_r) = [X]_r.\qedhere
		\end{eqnarray*}
	\end{proof}
	\begin{thm}\label{LAX}
		No permutation of the integers can induce a nontrivial automorphism of $\mc D_r$, for any reducibility $\le_r$ between $\le_1$ and $\le_T$.
	\end{thm}
	\begin{proof}
		Suppose $\theta:\omega\to\omega$ is a permutation (bijection) and that $\theta$ induces an automorphism of $\mc D_r$.
		Thus, letting $\Pi(A)=A\circ\theta$ and $\pi([A]_r)=[\Pi(A)]_r$, $\pi$ is well-defined and is an automorphism of $\mc D_r$.

		Recall that $S$ from Definition~\ref{waffle} is simply the successor function given by $S(n)=n+1$.
		For any $B$, let $A=B\circ\theta^{-1}\circ S$ (so $x\in A$ iff $S(x)\in B\circ\theta^{-1}$).
		We have $A\le_1 {B\circ\theta^{-1}}$ and so by assumption $A\circ\theta\le_T B\circ\theta^{-1}\circ\theta=B$.
		This gives $(B\circ\theta^{-1}\circ S)\circ\theta = \Phi^{B}$ for some Turing functional $\Phi$.
		Let $g=\theta^{-1}\circ S\circ\theta$.
		Since for each $B$ there exists such a $\Phi$, there must be some $\Phi$ such that the $G_\delta$ set
		\[
			\{B: B\circ g = \Phi^B\}
		\]
		is nonmeager.
		By Lemma~\ref{rumble}, $g$ is computable.
		By Lemma~\ref{waffle}, $\theta$ is computable.

		But this means that for any $A$, $A\circ\theta\le_1 A$, so that the map $\Pi$ is everywhere-decreasing: $\Pi(X)\le_1 X$, and in particular
		\begin{equation}\label{moderator}
			\pi([X]_r)\le [X]_r
		\end{equation} where $\le$ is the ordering of $\mc D_r$.
		By Lemma~\ref{fromReferee}, $\pi^{-1}$ is also induced by a permutation, namely $\theta^{-1}$.
		Applying (\ref{moderator}) to $\theta^{-1}$ we get
		\[
			\pi^{-1}([X]_r)\le [X]_r
		\]
		for all $X$, and so $\pi([X])_r = [X]_r$ for all $X$.
	\end{proof}

	\section{Excluding bi-uniformly $E_0$-invariant homeomorphisms}
		Making computability-theoretic uniformity assumptions is one way to rule out certain possible Turing automorphisms.
		We will instead focus on uniformity of a simpler, combinatorial kind.

		Let $[a,\infty)=\{n\in\omega: a\le n\}$.
		As usual we write $X=^* Y$ if $\{n\in\omega:X(n)\ne Y(n)\}$ is finite.
		The equivalence relation $=^*$ is also known as $E_0$.
		We refine this to
		\[
			X=^*_n Y \quad\Longleftrightarrow\quad X\restrict [n,\infty) = Y\restrict [n,\infty).
		\]
		As a relation, $=^*$ is the union $\bigcup_n =^*_n$.
		A map $F$ is $E_0$-invariant, or an $E_0$-endomorphism, if for all $X$, $Y$, if $X=^* Y$ then $F(X)=^* F(Y)$.
		In terms of the refinements $=^*_n$, this means that for all $X$, $Y$, and $a$, if $X=^*_a Y$ then there is an $b$ such that $F(X)=^*_b F(Y)$.
		If this $b$ only depends on $a$ we have a certain uniformity:
		\begin{df}\label{uniformModFinite}
			A function $F:2^\omega\to 2^\omega$ is a \emph{uniformly $E_0$-invariant} if
			for each $a$ there is a $b$ such that for all $X$, $Y$,
			if $X=^*_a Y$ then $F(X)=^*_b F(Y)$.
			If $F$ is invertible and both $F$ and $F^{-1}$ are uniformly $E_0$-invariant then $F$ is said to be a \emph{bi-uniformly $E_0$-invariant}.
			%{\color{red}rewrite more to talk about $E_0$ instead of mod finite}
		\end{df}
		Some continuous maps $F:2^\omega\to 2^\omega$ induce maps $\tilde F:{{2^\omega}/{=^*}}\to {{2^\omega}/{=^*}}$ but do not have the uniform property:
		\begin{exa}\label{partialThxToReferee}
			Let $F(A)=A\setminus\min(A)$, if $A\ne\emptyset$, and $F(A)=A$ otherwise.
			We shall show that for $a=1$ there is no $b$ as required in Definition~\ref{uniformModFinite}.
			Let $b$ be given.
			Let $Y$ be any set with $0<\min(Y)<\infty$ and $\min(Y)\ge b$, and let $X=Y\cup\{0\}$.
			Thus $X=^*_1 Y$.
			We compute that $F(X)=Y$ and $F(Y)=Y\setminus\min(Y)$.
			Suppose $F(X)=^*_b F(Y)$. Then $\min(Y)< b$, a contradiction.
			Thus $F$ is not uniformly $E_0$-invariant.

			On the other hand, $F$ is $E_0$-invariant.
			Indeed, let $d_H:2^\omega\times 2^\omega\to \omega+1$ denote the Hamming distance function,
			whereby $d_H(X,Y)=\abs{\{n:X(n)\ne Y(n)\}}$.
			We have $d_H(F(X),X)\le 1$, so that if $X=^*_a Y$
			then
			\[
				d_H(F(X),F(Y))\le d_H(F(X),X)+d_H(X,Y)+d_H(Y,F(Y))\le a+2
			\]
			and hence $F(X)=^* F(Y)$.
			%(Here we are using the observation that if $X$ and $Y$ differ in at most $u$ positions, and $Y$ and $Z$ differ in at most $v$ positions,
			%then $X$ and $Z$ differ in at most $u+v$ positions.)
		\end{exa}
		\begin{lem}\label{proofFixedByReferee}
			If $F(A)=A\circ f$ for a permutation $f:\omega\to\omega$ then $F$ is uniformly $E_0$-invariant.
		\end{lem}
		\begin{proof}
			Let $a\in\omega$. Let $b=\max\{f^{-1}(m): m<a\}$.
			Suppose $X=^*_a Y$, i.e., $X(m)=Y(m)$ for all $m\ge a$.
			Let $n> b$ and let $m$ be such that $n=f^{-1}(m)$.
			By definition of $b$, we have $m\ge a$, so that
			\[
				X\circ f(n) = X(f(f^{-1}(m))) = X(m) = Y(m) = Y(f(f^{-1}(m))) = Y\circ f(n),
			\]
			giving $X\circ f=^*_{b+1} Y\circ f$.
		\end{proof}

		Define the ${\star}$ operator by $f^{\star}(A)(n)=A(f(n))$.
		Again, recall the successor function $S$ from Definition~\ref{waffle}.
		\begin{lem}\label{belgian}
			Suppose $\Theta:2^\omega\to 2^\omega$ is a homeomorphism such that
			the function $\Theta^{-1}\circ S^{\star}\circ\Theta$ is computable.
			Then $\Theta$ is computable.
		\end{lem}
		\begin{proof}
			Let $\pi_n:\omega\to\omega$ be the constant $n$ function.
			For $\pi_n^{\star}:2^\omega\to 2^\omega$, note that $\pi_n^{\star}(A)(u)=(A\circ\pi_n)(u)=A(n)$,
			so $\pi_n^{\star}(A)\in\{0^\omega, 1^\omega\}$.
			Note that
			\begin{eqnarray*}
				(\pi_n^{\star}\circ S^{\star})(A)(u)
				&=&\pi_n^{\star}(S^{\star}(A))(u)
				=((S^{\star}(A))\circ\pi_n)(u)
				=(S^{\star}(A))(n)
			\\
				&=&A(S(n))
				=A(n+1)
				=(A\circ\pi_{n+1})(u)
				=(\pi_{n+1}^{\star}(A))(u)
			\end{eqnarray*}
			so $\pi_n^{\star}\circ S^{\star}=\pi_{n+1}^{\star}$.
			Also,
			$S\circ \pi_n = \pi_{n+1}$.
			Let
			\(
				\Phi = \Theta^{-1}\circ S^{\star}\circ\Theta.
			\)
			By assumption, $\Phi$ is computable.
			Then
			\begin{eqnarray*}
				\Theta\circ\Phi &=& S^{\star}\circ\Theta,\\
				%\pi_n^{\star}\circ\Theta\circ\Phi &=& \pi_n^{\star}\circ S^{\star}\circ\Theta\\
				%\pi_n^{\star}\circ\Theta\circ\Phi &=& \pi_{n+1}^{\star}\circ\Theta\\
				\text{hence}\quad(\pi_n^{\star}\circ\Theta)\circ\Phi &=& \pi_{n+1}^{\star}\circ\Theta.
			\end{eqnarray*}
			Since homeomorphisms have finite use, $\pi_0^{\star}\circ\Theta$ is just a finite amount of information.
			Thus, we can recursively compute 
			$\pi_{n+1}^{\star}\circ\Theta$ this way. For a concrete case, the reader may wish to inspect Example~\ref{indProc}.
		\end{proof}
		\begin{df}
			For a real \(X\) and a string \(\sigma\) of length \(n\),
			\[
				(\sigma\searrow X)(n)  = 
				\begin{cases}
					\sigma(n) & \text{if \(n<|\sigma|\),} \\
					X(n) & \text{otherwise.} 
				\end{cases}
			\]
		\end{df}

		\begin{lem}\label{crspToRumble}
			Suppose $F:2^\omega\to 2^\omega$ is a uniformly $E_0$-invariant continuous function.
			Suppose $F(X)=\Phi^X$ is forced above $\sigma\in 2^{<\omega}$, where $\Phi$ is a Turing functional.
			Then $F$ is computable.
		\end{lem}
		\begin{proof}
			Indeed, let $\sigma$ force $F(X)=\Phi^X$.
			Let $a=\abs{\sigma}$.
			Let $b$ be as in Definition~\ref{uniformModFinite}.
			Let the truth tables for $F$ for $n<b$ be given as a finite database.
			For $n\ge b$,
			\[
				F^X(n)=F^{\sigma\searrow X}(n)=\Phi^{\sigma\searrow X}(n).\qedhere
			\]
		\end{proof}
		\begin{thm}
			Let $\pi$ be an automorphism ot $\mc D_r$, for any reducibility $\le_r$ between $\le_1$ and $\le_T$.
			Suppose that $\pi$ is induced by a homeomorphism $\Theta$ of $2^\omega$ that is a bi-uniform $E_0$-isomorphism\footnote{
				The non-redundancy of the notion of bi-uniformly $E_0$-invariant Cantor homeomorphism has been pointed out by Salo~\cite{364844}.
			}.
			Then $\Theta$ is computable and $\pi$ is trivial.
		\end{thm}
		\begin{proof}
			The overall proof strategy mirrors that for Theorem~\ref{LAX}.
			Suppose that letting $\pi([A]_r)=[\Theta^A]_r$ makes $\pi$ well-defined and makes $\pi$ an automorphism of $\mc D_r$.

			Again, let $S$ be the successor function, $S(n)=n+1$.
			For any $B$, let $x\in A$ iff $S(x)\in \Theta^{-1}(B)$.
			We have $A\le_1 \Theta^{-1}(B)$ and so by assumption $\Theta(A)\le_T \Theta(\Theta^{-1}(B))=B$.
			This gives $\Theta^A = \Phi^{B}$ for some Turing functional $\Phi$.
			Let $\Gamma=\Theta\circ S^{\star}\circ\Theta^{-1}$, which is also uniformly $E_0$-invariant.
			Since for each $B$ there exists such a $\Phi$, there must be some $\Phi$ such that the $G_\delta$ set
			\[
				\{B: \Gamma(B) = \Phi^B\}
			\]
			is nonmeager.
			By Lemma~\ref{crspToRumble}, $\Gamma$ is computable.
			By Lemma~\ref{belgian}, $\Theta$ is computable.

			But this means that for any $A$, $\Theta(A)\le_{tt} A$, and in particular
			\begin{equation}\label{moderator2}
				\pi([X]_r)\le [X]_r
			\end{equation} where $\le$ is the ordering of $\mc D_r$.
			Now by assumption, $\pi^{-1}$ is also induced by a homeomorphism, namely $\Theta^{-1}$ and so applying
			(\ref{moderator2}) to $\Theta^{-1}$ we get
			\[
				\pi^{-1}([X]_r)\le [X]_r
			\]
			for all $X$, and so $\pi([X])_r = [X]_r$ for all $X$.
		\end{proof}

		\begin{exa}[The inductive procedure in Lemma~\ref{belgian}.]\label{indProc}
			Suppose $\Phi$ is the truth table reduction given by $\Phi^A(n)=A(2n)\cdot A(2n+1)$ for all $A$ and $n$.
			Suppose we know the first truth-table for $\Theta$, in that we know that $\Theta^A(0)=A(2)\to A(3)$ for all $A$.
			Then
			\begin{eqnarray*}
				\Theta^A(1)=\Theta^{\Phi^A}(0) &=& \Phi^A(2)\to\Phi^A(3)\\
				&=&A(4)A(5)\to A(6)A(7).
			\end{eqnarray*}
			Next,
			\begin{eqnarray*}
				 \Theta^A(2) &=& \Theta^{\Phi^A}(1)\\
				 &=& \Phi^A(4)\Phi^A(5)\to \Phi^A(6)\Phi^A(7)\\
				 &=& A(8)A(9)A(10)A(11)\to A(12)A(13)A(14)A(15).
			\end{eqnarray*}
		\end{exa}

		\begin{rem}
			Woodin mentioned on June 6, 2019 that he and Slaman may have shown the following result in unpublished work from the 1990s:
			Each automorphism of $\mathcal D_a$, the degrees of arithmetical reducibility, is represented by a continuous function (outright).
			This gives some extra interest in a possible future $\mathcal D_a$ version of our results.
		\end{rem}
	\newpage
	
	\bibliographystyle{plain}
	\bibliography{answer-schweber}
\end{document}